\documentclass[onefignum,onetabnum,final]{siamart190516}

\usepackage{lipsum}
\usepackage{amsfonts}
\usepackage{graphicx}
\usepackage{epstopdf}
\usepackage{algorithm}
\usepackage{algpseudocode}
\ifpdf
  \DeclareGraphicsExtensions{.eps,.pdf,.png,.jpg}
\else
  \DeclareGraphicsExtensions{.eps}
\fi

\newsiamremark{remark}{Remark}
\newsiamremark{hypothesis}{Hypothesis}
\crefname{hypothesis}{Hypothesis}{Hypotheses}
\newsiamthm{claim}{Claim}

\usepackage{pifont,latexsym,ifthen,calc}
\usepackage{amssymb,amsbsy,amsmath}
\usepackage[errorshow]{tracefnt}
\usepackage{hyperref}
\usepackage{enumerate}

\headers{The displacement structure of the Gram matrix}{K. Gumerov, S. Rigg, R. M. Slevinsky}

\title{Fast measure modification of orthogonal polynomials via matrices with displacement structure\thanks{Funding: This work is supported by the Natural Sciences and Engineering Research Council of Canada, through a Discovery Grant (RGPIN-2024-05480).}}
\author{Karim Gumerov\thanks{Faculty of Engineering, University of Manitoba, Winnipeg, MB, Canada, (\email{gumerovk@myumanitoba.ca})} \and Samantha Rigg\thanks{Faculty of Engineering, University of Manitoba, Winnipeg, MB, Canada, (\email{riggs2@myumanitoba.ca})} \and Richard Mika\"el Slevinsky\thanks{Department of Mathematics, University of Manitoba, Winnipeg, MB, Canada, (\email{Richard.Slevinsky@umanitoba.ca})}}

\usepackage{amsopn}

\def\DD{\mathcal{D}}
\def\norm#1{\left\|{#1}\right\|}

\ifpdf
\hypersetup{
  pdftitle={FastTransformsByDisplacementStructure},
  pdfauthor={K. Gumerov, S. Rigg, R. M. Slevinsky}
}
\fi

\usepackage{todonotes}

\def\rank{\operatorname{rank}}

\def\rmspmatrix#1{\begin{pmatrix}#1\end{pmatrix}}
\def\mapengine#1,#2.{\mapfunction{#1}\ifx\void#2\else\mapengine #2.\fi }
\def\map[#1]{\mapengine #1,\void.}
\def\mapenginesep_#1#2,#3.{\mapfunction{#2}\ifx\void#3\else#1\mapengine #3.\fi }
\def\mapsep_#1[#2]{\mapenginesep_{#1}#2,\void.}
\def\bvect[#1,#2]{
{
\def\mapfunction##1{\ | \  ##1}
	\rmspmatrix{
		 \,#1\map[#2]\,
	}
}
}
\def\bs#1{\boldsymbol{#1}}
\def\bbR{{\mathbb R}}

\begin{document}

\maketitle

% REQUIRED
\begin{abstract}
It is well known that matrices with low Hessenberg-structured displacement rank enjoy fast algorithms for certain matrix factorizations. We show how $n\times n$ principal finite sections of the Gram matrix for the orthogonal polynomial measure modification problem has such a displacement structure, unlocking a collection of fast algorithms for computing connection coefficients (as the upper-triangular Cholesky factor) between a known orthogonal polynomial family and the modified family. In general, the ${\cal O}(n^3)$ complexity is reduced to ${\cal O}(n^2)$, and if the symmetric Gram matrix has upper and lower bandwidth b, then the ${\cal O}(b^2n)$ complexity for a banded Cholesky factorization is reduced to ${\cal O}(b n)$. In the case of modified Chebyshev polynomials, we show that the Gram matrix is a symmetric Toeplitz-plus-Hankel matrix, and if the modified Chebyshev moments decay algebraically, then a hierarchical off-diagonal low-rank structure is observed in the Gram matrix, enabling a further reduction in the complexity of an approximate Cholesky factorization powered by randomized numerical linear algebra.
\end{abstract}

% REQUIRED

% REQUIRED

\section{Introduction}

Let $\lambda$ be a positive Borel measure on the real line whose support contains an infinite number of points and has finite moments:
\[
-\infty < \int_\bbR x^n{\rm\,d}\lambda(x) < \infty,\quad\forall n\in\mathbb{N}_0.
\]
Then there exists a family of orthogonal polynomials ${\bf P}(x) = (p_0(x) \quad p_1(x) \quad p_2(x) \quad \cdots)$ with respect to the inner-product:
\[
\langle f, g\rangle_\lambda = \int_\bbR f(x) g(x){\rm\,d}\lambda(x).
\]
That is, $\langle p_m, p_n\rangle_w = k_n\delta_{m,n}$, where $k_n > 0$. Moreover, if $k_n\equiv1$, then we say that ${\bf P}(x)$ are orthonormal. This inner-product induces a norm, $\|f\|_\lambda^2 = \langle f, f\rangle_\lambda$, and the Hilbert space $L^2(\bbR,{\rm\,d}\lambda)$.

Orthonormal polynomials are important in numerical analysis for their ability to represent $f\in L^2(\bbR,{\rm\,d}\lambda)$ isomorphically as an $\ell^2$ infinite vector of expansion coefficients via:
\[
f(x) = \sum_{k=0}^\infty \langle f, p_k\rangle_\lambda\,p_k(x) =: {\bf P}(x) \bs{f}.
\]
Given a known family of orthogonal polynomials, ${\bf P}(x)$, we wish to describe efficient algorithms to compute the orthonormal polynomials ${\bf Q}(x)$ in $L^2(\bbR,{\rm\,d}\mu)$. Our numerical apparatus is centred on the computation of the connection coefficients between the original and the modified families: 
\[
{\bf P}(x) = {\bf Q}(x)R.
\]
That is, the computation of the coefficients $R_{k,n}$ (the entries of the upper-triangular matrix $R$) such that:
\[
p_n(x) = \sum_{k=0}^n q_k(x)R_{k,n}.
\]

Every family of orthogonal polynomials has an irreducible infinite tridiagonal matrix (the transpose of the Jacobi matrix) that implements multiplication by $x$:
\[
x{\bf P}(x) = {\bf P}(x) X_P.
\]
If ${\bf P}(x)$ is orthonormal, then $X_P$ is symmetric. As a characterizing object, it is essential to provide methods to compute the matrix $X_Q$ in $x{\bf Q}(x) = {\bf Q}(x)X_Q$. Indeed, a result due to Gautschi~\cite{Gautschi-24-245-70} states that $RX_P = X_QR$. Since $X_P$ and $X_Q$ are both tridiagonal matrices, the computation of $X_Q$ involves only the main diagonal and the first super-diagonal of the connection coefficients~\cite{Gutleb-Olver-Slevinsky-1-24}.

\begin{definition}
The Gram matrix is defined by:
\begin{equation}\label{eq:GramMatrix}
W_P = \int_\bbR {\bf P}(x)^\top {\bf P}(x) {\rm\,d}\mu(x).
\end{equation}
\end{definition}

\begin{proposition}\label{proposition:Gram}
Properties of the Gram matrix include:
\begin{enumerate}
\item $W_P$ is symmetric and positive-definite;
\item If $W_P = R^\top R$ is a Cholesky factorization, then ${\bf P}(x) = {\bf Q}(x) R$ and the converse is true; and,
\item $X_P^\top W_P - W_PX_P = 0$.
\end{enumerate}
\end{proposition}
\begin{proof}
Properties 1 and 2 are proved in~\cite{Gutleb-Olver-Slevinsky-1-24}. The last property follows from:
\begin{align*}
\int_\bbR {\bf P}(x)^\top x{\bf P}(x) {\rm\,d}\mu(x) & = \int_\bbR {\bf P}(x)^\top {\bf P}(x) {\rm\,d}\mu(x) X_P = W_PX_P,\\
& = X_P^\top \int_\bbR {\bf P}(x)^\top {\bf P}(x) {\rm\,d}\mu(x) = X_P^\top W_P.
\end{align*}
\end{proof}
The fast algorithms in this work are a consequence of this third property.

\section{Fast Cholesky decomposition of the Gram matrix}

\begin{definition}
Let $e_n$ be the $n^{th}$ canonical basis vector.
\end{definition}

\begin{definition}
Let $P_n\in\bbR^{n\times\infty}$ denote the projection operator:
\begin{equation}\label{eq:canonicalorthogonalprojection}
P_n = \begin{pmatrix} I_n & 0\end{pmatrix},
\end{equation}
where $I_n$ is the $n\times n$ identity.
\end{definition}

\begin{definition}
The modified OP moments of $\mu$ are defined by:
\[
\bs{\mu} = \int_\bbR {\bf P}(x)^\top{\rm\,d}\mu(x).
\]
\end{definition}

\begin{theorem}[Gautschi~\cite{Gautschi-24-245-70}]\label{theorem:momentstoX}
Let the first $2n-1$ modified OP moments of $\mu$ be known. Then these uniquely define $P_n W_P P_n^\top$ and thereby $P_n R P_n^\top$ and $P_{n-1} X_Q P_{n-1}^\top$.
\end{theorem}
\begin{proof}
The first column of the Gram matrix is equal to the constant $P_0(x)$ times the modified OP moments. In the standard normalization, all classical OPs satisfy $P_0(x) = 1$, but if ${\bf P}(x)$ are orthonormal, the constant may be different~\cite{NIST:DLMF}. Thus the first $2n-1$ entries of the first column of $W_P$ are known. In general, Property 3 in Proposition~\ref{proposition:Gram} defines a five-term recurrence relation about the entry $m,n$ as follows, dropping subscripts for clarity:
\begin{equation}
\begin{aligned}
& X_{m-1, m}W_{m-1, n} + X_{m, m}W_{m, n} + X_{m+1, m}W_{m+1, n}\\
& = W_{m, n+1}X_{n+1, n} + W_{m, n}X_{n, n} + W_{m, n-1}X_{n-1, n}.\label{eq:MomentRecurrence}
\end{aligned}
\end{equation}
For $n=1$, the term $W_{m, n-1}X_{n-1,n}$ is omitted. This recurrence relation allows the first $2n-1$ entries in the first column to define the first $2n-2$ entries in the second column, and more generally the first $2n-m$ entries in the $m^{\rm th}$ column. Thus $P_n W_P P_n^\top$ is defined, its Cholesky factorization exists and is unique, and the computation $P_{n-1}X_QP_{n-1}^\top = (P_{n-1}RP_n^\top) (P_nX_P P_n^\top) (P_n R^{-1} P_{n-1}^\top)$ is well-defined.
\end{proof}

\begin{remark}\label{remark:fastGramfill} Eq.~\eqref{eq:MomentRecurrence} allows for all entries in the $n\times n$ principal finite section of the Gram matrix to be computed in ${\cal O}(n^2)$ flops. If, however, every entry in the moment vector past the first $b+1$ is zero, then $W_P$ is banded with upper and lower bandwidths $b$, and the same section can be computed in ${\cal O}(bn)$ flops.
\end{remark}

Property 3 in Proposition~\ref{proposition:Gram} can be restated in finite dimensions as follows.

\begin{lemma}\label{lemma:finitedimensionalmatrixequation}
%If $G = (P_nX_P^\top e_{n+1} | -P_nW_Pe_{n+1})\in\bbR^{n\times 2}$ and $J = \begin{pmatrix} 0 & 1\\ -1 & 0\end{pmatrix}$, then:
If $G = (e_n | -P_nW_Pe_{n+1}(X_P)_{n+1,n})\in\bbR^{n\times 2}$ and $J = \begin{pmatrix} 0 & 1\\ -1 & 0\end{pmatrix}$, then:
\begin{equation}\label{eq:finitedimensionalmatrixequation}
(P_n X_P^\top P_n^\top) (P_n W_P P_n^\top) - (P_nW_PP_n^\top) (P_nX_PP_n^\top) = GJG^\top.
\end{equation}
\end{lemma}
\begin{proof}
The skew-symmetry of the left-hand side of Eq.~\eqref{eq:finitedimensionalmatrixequation} justifies the storage of the right-hand side in skew-symmetric form. Since $X_P$ is tridiagonal, projection of the second half of the infinite-dimensional matrix equation reads:
\begin{align*}
P_nW_PX_PP_n^\top & = P_nW_PP_{n+1}^\top P_{n+1}X_PP_n^\top,\\
& = P_nW_P\left[P_n^\top P_n + \left(P_{n+1}^\top P_{n+1} - P_n^\top P_n\right) \right]X_PP_n^\top,\\
& = (P_nW_PP_n^\top) (P_nX_PP_n^\top) + (P_nW_Pe_{n+1})(P_n X_P^\top e_{n+1})^\top,\\
& = (P_nW_PP_n^\top) (P_nX_PP_n^\top) + (P_nW_Pe_{n+1})(X_P)_{n+1, n} e_n^\top.
\end{align*}
\end{proof}

\begin{remark}\label{remark:generators}
While Lemma~\ref{lemma:finitedimensionalmatrixequation} proves that the $n\times n$ principal finite section of the Gram matrix satisfies a skew-symmetric rank-$2$ displacement equation with $P_nX_PP_n^\top$, it does so by establishing an impractical form of the generators, $G$. The above formulas for $G$ require the $n+1^{\rm st}$ columns of $W_P$ and $X_P$, which may not be readily available when already working with truncations. As a remedy, we observe that the nonzero entries in the difference:
\[
(P_n X_P^\top P_n^\top) (P_n W_P P_n^\top) - (P_nW_PP_n^\top) (P_nX_PP_n^\top),
\]
live only in the last row and the last column. Thus, an alternative form for $G$ reads:
\[
G = \Big(e_n \Big| (P_nW_Pe_{n-1})(X_P)_{n-1,n} + (P_nW_Pe_n)(X_P)_{n,n} - (P_n X_P^\top P_n^\top) (P_n W_P e_n) \Big),
\]
which follows from:
\begin{align*}
& \left[(P_n X_P^\top P_n^\top) (P_n W_P P_n^\top) - (P_nW_PP_n^\top) (P_nX_PP_n^\top)\right] e_n,\\
& = (P_n X_P^\top P_n^\top) (P_n W_P e_n) - (P_nW_PP_n^\top) (P_nX_Pe_n),\\
& = (P_n X_P^\top P_n^\top) (P_n W_P e_n) - (P_nW_PP_n^\top) \left[(X_P)_{n-1,n}e_{n-1} + (X_P)_{n,n}e_n\right],\\
& = (P_n X_P^\top P_n^\top) (P_n W_P e_n) - (P_nW_Pe_{n-1})(X_P)_{n-1,n} - (P_nW_Pe_n)(X_P)_{n,n}.
\end{align*}
\end{remark}

Fortunately, much has been done~\cite{Gohberg-Kailath-Olshevsky-64-1557-95,Heining-Olshevsky-281-3-01,Kressner-Thesis-01} to analyze matrix equations with displacement structure such as Eq.~\eqref{eq:finitedimensionalmatrixequation}, and it turns out that its Cholesky factorization can be computed in ${\cal O}(n^2)$ flops instead of ${\cal O}(n^3)$. This algorithm uses the low displacement rank to perform the symmetric transformations implied by a Cholesky factorization and apply them to the matrix of {\em generators}, $G$, and to $P_nX_PP_n^\top$ rather than working directly on $P_nW_PP_n^\top$.

The following lemma is special case of~\cite[Lemma 2.1]{Heining-Olshevsky-281-3-01} applied to Sylvester-type matrix equations with tridiagonal-plus-first-row Hessenberg structure.

\begin{lemma}\label{lemma:inductionstep}
Suppose we have a Sylvester-type displacement equation for a symmetric and positive-definite $W_1\in\bbR^{n\times n}$:
\[
X_1^\top W_1 - W_1 X_1 = G_1JG_1^\top,
\]
where $G_1\in\bbR^{n\times 2}$ and $J\in\bbR^{2\times 2}$, and where $X_1\in\bbR^{n\times n}$ is an irreducible tridiagonal matrix with nonzero entries in the first row. That is:
\[
X_1 = \begin{pmatrix} \alpha & \beta e_1^\top\\ \gamma e_1 & \hat{X}_1\end{pmatrix} + e_1r_1^\top,
\]
where $\hat{X}_1\in\bbR^{(n-1)\times(n-1)}$ is a tridiagonal matrix. Furthermore, let $W_1 = \begin{pmatrix} d_1 & l_1^\top\\ l_1 & \hat{W}_1\end{pmatrix}$ and let $L_1 = \begin{pmatrix} \sqrt{d_1} & 0\\ \frac{l_1}{\sqrt{d_1}} & I\end{pmatrix}$ so that:
\[
W_1 = L_1\begin{pmatrix} 1 & 0\\ 0 & W_2\end{pmatrix}L_1^\top,
\]
is a partial Cholesky factorization, where $W_2 = \hat{W}_1 - \frac{1}{d_1}l_1l_1^\top$ is the Schur complement. If $d_1\ne 0$, then the Schur complement satisfies another Sylvester-type displacement equation of the same form:
\[
X_2^\top W_2 - W_2 X_2 = G_2JG_2^\top,
\]
where:
\[
X_2 = \hat{X}_1 + e_1r_2^\top,\qquad r_2 = -\frac{\gamma}{d_1}l_1,
\]
and if $G_1 = \begin{pmatrix} g_1\\ \hat{G}_1\end{pmatrix}$, then:
\[
G_2 = \hat{G}_1 - \frac{l_1g_1}{d_1}.
\]
\end{lemma}
\begin{proof}
The original matrix equation is equivalent to:
\[
X_1^\top L_1\begin{pmatrix} 1\\ & W_2\end{pmatrix}L_1^\top - L_1\begin{pmatrix} 1\\ & W_2\end{pmatrix}L_1^\top X_1 = G_1JG_1^\top.
\]
After applying $L_1^{-1}$ from the left and $L_1^{-\top}$ from the right, we have:
\[
L_1^{-1}X_1^\top L_1\begin{pmatrix} 1\\ & W_2\end{pmatrix} - \begin{pmatrix} 1\\ & W_2\end{pmatrix}L_1^\top X_1L_1^{-\top} = L_1^{-1}G_1JG_1^\top L_1^{-\top}.
\]
Computing the product:
\begin{align*}
L_1^\top X_1 L_1^{-\top} & = \begin{pmatrix} \sqrt{d_1} & \frac{l_1^\top}{\sqrt{d_1}}\\ 0 & I\end{pmatrix} \left[\begin{pmatrix} \alpha & \beta e_1^\top\\ \gamma e_1 & \hat{X}_1\end{pmatrix} + e_1r_1^\top\right] \begin{pmatrix} \frac{1}{\sqrt{d_1}} & -\frac{l_1^\top}{d_1}\\ 0 & I\end{pmatrix},\\
& = \begin{pmatrix} \times & \times\\ \times & \hat{X}_1-\frac{\gamma}{d_1}e_1l_1^\top\end{pmatrix},
\end{align*}
where the $\times$ indicate entries that are not of interest, the second term in the original matrix equation becomes:
\[
\begin{pmatrix} 1\\ & W_2\end{pmatrix}L_1^\top X_1L_1^{-\top} = \begin{pmatrix} \times & \times\\ \times & W_2X_2\end{pmatrix}.
\]
Finally,
\[
\begin{pmatrix} \times\\ G_2\end{pmatrix} = L_1^{-1} G_1.
\]
\end{proof}

This lemma nearly forms the full induction step from a matrix equation of size $n\times n$ to one of size $(n-1)\times(n-1)$. To begin the computations, the first column of $W_1$ is required to form the partial Cholesky factor $L_1$. Indeed, for the Gram matrix, this first column is simply proportional to the modified OP moments, but to continue the inductive process, we must also compute the first column of $W_2$. For Sylvester-type matrix equations with Hessenberg structure, this is known as the {\em second-column problem}~\cite{Heining-Olshevsky-281-3-01}, because the first column of $W_2$ may be constructed from $L_1$ and the second column of $W_1$. This second column is computed by applying the matrix equation to $e_1$:
\begin{align*}
X_1^\top W_1e_1 - W_1 X_1e_1 & = G_1JG_1^\top e_1,\\
X_1^\top W_1e_1 - W_1 \left[\begin{pmatrix} \alpha & \beta e_1^\top\\ \gamma e_1 & \hat{X}_1\end{pmatrix} + e_1r_1^\top\right]e_1 & = G_1Jg_1,
\end{align*}
or:
\[
\gamma W_1e_2 = \left[X_1^\top - (\alpha + r_1^\top e_1)\right] W_1e_1 - G_1Jg_1.
\]
Since $X_1$ is assumed to be irreducible, $\gamma\ne0$.

Algorithm~\ref{algorithm:fastCholesky} describes the fast Cholesky factorization of the principal finite section of the Gram matrix as a consequence of Lemmas~\ref{lemma:finitedimensionalmatrixequation} and~\ref{lemma:inductionstep}.

\begin{algorithm}\label{algorithm:fastCholesky}
\caption{Fast Cholesky factorization of symmetric positive-definite $W\in\bbR^{n\times n}$.}
\begin{algorithmic}
\Require $X\in\bbR^{n\times n}$ is an irreducible tridiagonal matrix and $X^\top W - W X = GJG^\top$.
\Function{Cholesky}{$W, X, G, J$}
\State $L = I_n$
\State $c = W_{:,1}$
\For {$k = 1:n-1$}
	\State $d = \sqrt{c_1}$
	\State $L_{k:n,k} = l = c / d$
	\State $\hat{c} = ((X^\top-X_{1,1}I)c - GJG_{1,:})/X_{2,1}$ \Comment{The second column.}
	\State $c = \hat{c}_{2:n-k+1} - (c_2/d) l_{2:n-k+1}$ \Comment{The next first column.}
	\State $X = X_{2:n-k+1,2:n-k+1} - (X_{2,1}/d) e_1 l_{2:n-k+1}^\top$
	\State $G = G_{2:n-k+1,:} - l_{2:n-k+1}G_{1,:}/d$
\EndFor
\State $L_{n,n} = \sqrt{c_1}$
\State \Return $L$
\EndFunction
\Ensure $W = LL^\top$.
\end{algorithmic}
\end{algorithm}

Surprisingly, Algorithm~\ref{algorithm:fastCholesky} appears symbolically identical to a fast $QR$ factorization of a Toeplitz-plus-Hankel matrix~\cite[Algorithm 2.31]{Kressner-Thesis-01} with rank-$2$ generators in this case instead of rank-$8$ generators for the $QR$ factorization. The relationship follows from the fact that the Chebyshev--Gram matrix, as described in \S~\ref{section:hierarchical}, is a special symmetric Toeplitz-plus-Hankel matrix whose defining constants are drawn from the same moment vector.

\subsection{Computing modified orthogonal polynomial moments}

By Theorem~\ref{theorem:momentstoX} and Remark~\ref{remark:fastGramfill}, entries of $P_nW_PP_n^\top$ may be computed by recurrence starting with the modified OP moments. In this section, we shall describe two strategies to compute these moments for classical OPs. That is, we shall discuss strategies in case ${\rm\,d}\mu(x) = w(x)\chi_{(a,b)}(x){\rm\,d}x$, where $\chi_I$ is the indicator function on the set $I$ and ${\bf P}(x)$ are classical orthogonal polynomials~\cite[\S 18]{NIST:DLMF} with respect to weight $w_c(x)$:
\[
\bs{\mu} = \bs{\mu}[w] = \int_a^b{\bf P}(x)^\top w(x){\rm\,d}x.
\]

We shall use the fact that the classical orthogonal polynomial weight $w_c(x)$ satisfies the Pearson differential equation $(\sigma w_c)' = \tau w_c$, where $\deg(\sigma)\le 2$ and $\deg(\tau)\le 1$, and there exists~\cite{Gutleb-Olver-Slevinsky-1-24,Olver-Slevinsky-Townsend-29-573-20} a banded differentiation matrix in which:
\[
{\cal D}{\bf P}(x) = {\bf P}'(x)D_P^{P'},
\]
banded raising (upper triangular) and lowering (lower triangular) matrices:
\[
{\bf P}(x) = {\bf P}'(x)R_P^{P'},\quad{\rm and}\quad\sigma {\bf P}'(x) = {\bf P}(x) L_{P'}^P,
\]
and a diagonal mass matrix:
\[
\int_a^b {\bf P}(x)^\top {\bf P}(x)w_c(x){\rm\,d}x = M_P.
\]
The upper bandwidth of $D_P^{P'}$ is $1$, and the upper bandwidth of $R_P^{P'}$ and the lower bandwidth of $L_{P'}^P$ are both equal to $\deg(\sigma)$.

\subsubsection{From differential equation to recurrence relation}

\begin{theorem}\label{theorem:weightODE}
Let $w(x)/w_c(x)\in L^2_{w_c}(a,b)$ and let $w$ satisfy the following differential equation:
\begin{equation}\label{eq:weightODE}
a(x)(\sigma w)' + b(x) w(x) = c(x),
\end{equation}
where $a(x)$ and $b(x)$ are polynomials with respective degrees $a$ and $b$ and the modified moments of $c$ are known. Then there exists a linear recurrence relation for the modified moments $\bs{\mu}[w]$ of length at most $\max\{2a+3, 2b+1\}$.
\end{theorem}
\begin{proof}
The $2$-norm convergent series $w(x) = w_c(x){\bf P}(x) \bs{w}$ also encodes a relationship between the modified moments $\bs{\mu}[w]$ and the coefficients $\bs{w}$. Multiplying by ${\bf P}(x)^\top$ and integrating, we find:
\[
\bs{\mu}[w] = M_P\bs{w}.
\]
Inserting this series into the differential equation:
\begin{align*}
& a(x){\cal D}\left[\sigma w_c{\bf P}(x)\bs{w}\right] + b(x)w_c(x){\bf P}(x)\bs{w}\\
& = a(x)w_c(x)\left\{\sigma {\bf P}'(x)D_P^{P'}\bs{w} + \tau(x){\bf P}(x)\bs{w}\right\} + b(x)w_c(x){\bf P}(x)\bs{w}\\
& = a(x)w_c(x){\bf P}(x)\left[L_{P'}^PD_P^{P'} + \tau(X_P)\right]\bs{w} + b(x)w_c(x){\bf P}(x)\bs{w}\\
& = w_c(x){\bf P}(x)\left\{a(X_P)\left[L_{P'}^PD_P^{P'} + \tau(X_P)\right] + b(X_P)\right\}\bs{w} = c(x).
\end{align*}
The recurrence is thus:
\[
M_P\left\{a(X_P)\left[L_{P'}^PD_P^{P'} + \tau(X_P)\right] + b(X_P)\right\}M_P^{-1}\bs{\mu}[w] = \bs{\mu}[c].
\]
\end{proof}

Table~\ref{table:weightODEs} illustrates two classical weights and many nonclassical weights that satisfy such a differential equation, which, when coupled with a sufficient number of initial conditions, leads to a finite-length recurrence relation for the modified moments. In this table, we use the shorthands $\displaystyle\bs{t}+x = \prod_{i=1}^k t_i+x$ and $\displaystyle|\bs{t}+x|^{\bs{\gamma}} = \prod_{i=1}^k |t_i+x|^{\gamma_i}$ to denote weights with $k$ algebraic factors.

\begin{table}[htp]
\caption{A collection of weights satisfying Eq.~\eqref{eq:weightODE}. In the upper part, $\sigma(x) = 1-x^2$ and in the lower part, $\sigma(x) = x$.}
\begin{center}
\begin{tabular}{c|c}
$w(x)$ & Differential equation\\
\hline
\rule{0pt}{3ex} $(1-x)^\alpha(1+x)^\beta$ & $(\sigma w)' + \left[\alpha-\beta+(\alpha+\beta+2)x\right]w = 0$\\
\rule{0pt}{3ex} $\log\left(\frac{2}{1-x}\right)(1-x)^\alpha(1+x)^\beta$ & $(\sigma w)' + \left[\alpha-\beta+(\alpha+\beta+2)x\right]w = (1-x)^\alpha(1+x)^{\beta+1}$\\
\rule{0pt}{3ex} $|\bs{t}+x|^{\bs{\gamma}}$ & $\displaystyle (\bs{t}+x)w' = w\sum_{i=1}^k\gamma_i\prod_{\substack{j=1\\j\ne i}}^k t_j+x$\\
\hline
\rule{0pt}{3ex} $x^\alpha e^{-x}$ & $(\sigma w)' + (x-\alpha-1)w = 0$\\
\rule{0pt}{3ex} $\log(t+x)x^\alpha e^{-x}$ & $(t+x)(\sigma w)' + (t+x)(x-\alpha-1)w = x^{\alpha+1}e^{-x}$\\
\rule{0pt}{3ex} $|\bs{t}+x|^{\bs{\gamma}}e^{-x}$ & $\displaystyle (\bs{t}+x)(w'+w) = w\sum_{i=1}^k\gamma_i\prod_{\substack{j=1\\j\ne i}}^k t_j+x$\\
\hline
\end{tabular}
\end{center}
\label{table:weightODEs}
\end{table}

Rather than being exhaustive, Table~\ref{table:weightODEs} is intended to be exemplary: algebraic powers satisfy a first-order linear differential equation with polynomial coefficients, logarithmic factors require inhomogeneous linear differential equations, but in Table~\ref{table:weightODEs} the moments of the inhomogeneities also satisfy recurrence relations. The differential equations are written as compactly as possible, and sometimes not exactly in the form of Eq~\eqref{eq:weightODE}, but by $(\sigma w)' = \sigma w' + \sigma'w$, it is easy to convert between representations.

\subsubsection{Simple function approximation}

Not every weight satisfies a differential equation with the conditions outlined above. Therefore, we present a computational framework based on simple function approximation.

\begin{definition}
Let $a = x_0 < x_1 < \cdots < x_n = b$ be a partition of the interval $[a,b]$ into a disjoint union of measurable sets. Then $s(x)$ is a simple function on $[a,b]$ if it is constant on every open subinterval of the partition. That is:
\[
s(x) = \sum_{k=1}^n s_k\chi_{(x_{k-1},x_k)}(x).
\]
\end{definition}

\begin{theorem}[Theorem 10.19 in~\cite{Apostol-81}]
Let $f\in L^1(I)$. For every $\epsilon>0$, there exists a simple function $s$ such that:
\[
\norm{f-s}_1 < \epsilon.
\]
\end{theorem}

\begin{theorem}\label{theorem:firstsimple}
Let $w$ be a simple function on $[a,b]$:
\[
w(x) = \sum_{k=1}^n w_k\chi_{(x_{k-1},x_k)}(x).
\]
Then:
\[
\bs{\mu}[w] = \left[(D_P^{P'})^+R_P^{P'}\right]^\top \sum_{k=1}^nw_k \left[{\bf P}(x_k)-{\bf P}(x_{k-1})\right]^\top.
\]
\end{theorem}
\begin{proof}
Let $(D_P^{P'})^+$ be the Moore--Penrose pseudoinverse of the classical orthogonal polynomial differentiation matrix $D_P^{P'}$~\cite{Gutleb-Olver-Slevinsky-1-24}. Then:
\[
\int_{\alpha}^\beta {\bf P}(x){\rm\,d}x = \left[{\bf P}(\beta)-{\bf P}(\alpha)\right](D_P^{P'})^+ R_P^{P'}.
\]
The identity follows from:
\[
\int_{\alpha}^\beta {\bf P}(x){\rm\,d}x = \int_{\alpha}^\beta {\bf P}'(x){\rm\,d}xR_P^{P'} = \int_{\alpha}^\beta {\cal D}{\bf P}(x){\rm\,d}x(D_P^{P'})^+ R_P^{P'}.
\]
\end{proof}

Theorem~\ref{theorem:firstsimple} requires the simple function be $0$ for sufficiently large $x$ on unbounded intervals, causing the approximation of the Gram matrix to be semi-definite --- which can be a numerical problem for Cholesky factorization. Therefore, we next describe a variant of this process useful on unbounded intervals.
\begin{theorem}\label{theorem:unboundedsimple}
Let $w$ be the pointwise product of a classical weight and a simple function on $[a,b]$:
\[
w(x) = w_c(x)\sum_{k=1}^n w_k\chi_{(x_{k-1},x_k)}(x).
\]
Then:
\[
\bs{\mu}[w] = (D_P^{P'})^+ \sum_{k=1}^nw_k \left[\sigma(x_{k-1})w_c(x_{k-1}){\bf P}'(x_{k-1})-\sigma(x_k)w_c(x_k){\bf P}'(x_k)\right]^\top.
\]
\end{theorem}
\begin{proof}
The transpose classical orthogonal polynomial differentiation matrix satisfies~\cite{Gutleb-Olver-Slevinsky-1-24}:
\[
(-{\cal D})[\sigma(x)w_c(x){\bf P}'(x)] = w_c(x){\bf P}(x) (D_P^{P'})^\top.
\]
It follows that:
\[
\int_{\alpha}^\beta {\bf P}(x)w_c(x){\rm\,d}x = \left[{\bf P}'(\alpha)\sigma(\alpha)w_c(\alpha) - {\bf P}'(\beta)\sigma(\beta)w_c(\beta)\right]\left[(D_P^{P'})^\top\right]^+.
\]
\end{proof}

\section{Numerical experiments}

Our numerical experiments are conducted on an iMac Pro (Early 2018) with a $2.3$ GHz Intel Xeon W-2191B with $128$ GB $2.67$ GHz DDR4 RAM. Our algorithms are available in Julia in {\tt FastTransforms.jl}~\cite{Slevinsky-GitHub-FastTransformsjl}.

\subsection{A single algebraic factor on $(-1,1)$}

In this example, we consider orthogonal polynomials with respect to the weight:
\[
w(x) = \frac{1}{\sqrt{1+\delta-x}},\quad{\rm for}\quad \delta>0.
\]
This problem is related to the half-range Chebyshev polynomials~\cite{Huybrechs-47-4326-10} but the inclusion of the parameter $\delta$ makes it a more interesting benchmark. The weight is analytic on $(-1,1)$ and its Legendre series is known in closed form~\cite[\S 18.12.11]{NIST:DLMF}:
\[
w(x) = \sqrt{\frac{2}{\rho}}\sum_{n=0}^\infty \frac{P_n(x)}{\rho^n}.
\]
Here, $\rho = 1+\delta+\sqrt{\delta^2+2\delta} > 1$ is the Bernstein ellipse parameter on which the singularity of the weight lives. By~\cite[Theorem 2.14]{Gutleb-Olver-Slevinsky-1-24}, suffice it to consider the truncation of the Legendre series as a proxy for the weight itself: given an $\epsilon > 0$, such as $\epsilon_{\rm mach} \approx 2.22\times10^{-16}$, if:
\[
b = \log_\rho\left(\frac{\sqrt{\delta}}{\epsilon}\frac{\sqrt{2/\rho}}{\rho-1}\right),
\]
then:
\[
\norm{w - \sqrt{\frac{2}{\rho}}\sum_{n=0}^b \frac{P_n(x)}{\rho^n}}_\infty \le \epsilon \norm{w}_\infty,
\]
and $w$ is to all intents and purposes a degree-$b$ polynomial. Moreover, using this polynomial results in a Legendre--Gram matrix with the same $\epsilon$-bound on the relative induced $2$-norm. This problem will allow us to distinguish the ${\cal O}(bn)$ algorithms in this work from the ${\cal O}(b^2n)$ algorithms of~\cite{Gutleb-Olver-Slevinsky-1-24}. Figure~\ref{fig:legendregenerating} illustrates the differences in error growth and calculation time between both algorithms. As the eigenvalues of the Legendre--Gram matrix are contained in the interval $[(2+\delta)^{-\frac{1}{2}},\delta^{-\frac{1}{2}}]$, corresponding to the extrema of $w(x)$, the $2$-norm condition number of the finite sections does not exhibit $n$-dependent growth. The mild algebraic growth observed in the error in the displacement algorithm, therefore, is an indication that it is less stable than a direct Cholesky factorization, consistent with previous experiments~\cite{Gohberg-Kailath-Olshevsky-64-1557-95}.

\begin{figure}[htbp]
\begin{center}
\begin{tabular}{cc}
\includegraphics[width=0.48\textwidth]{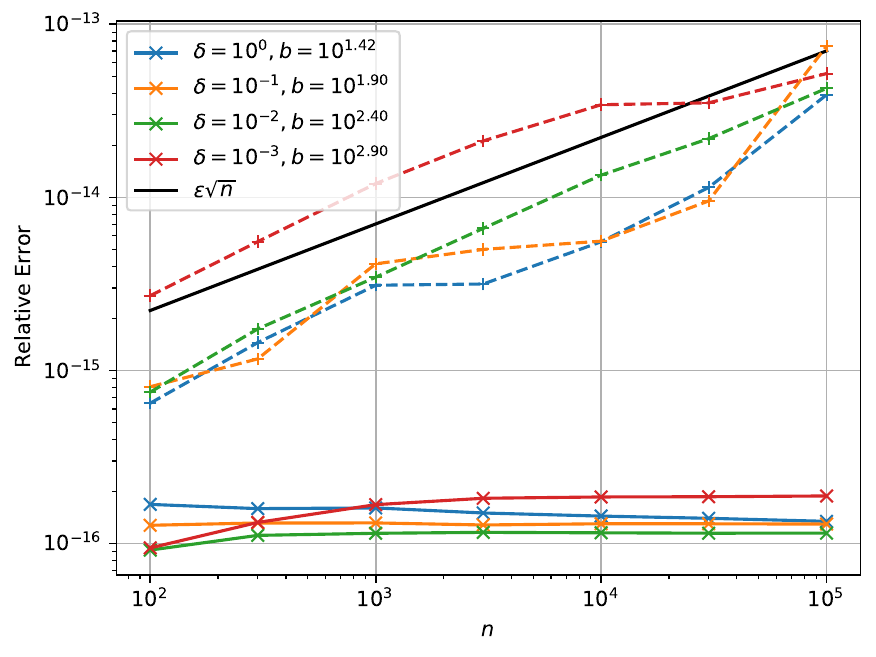}&
\includegraphics[width=0.48\textwidth]{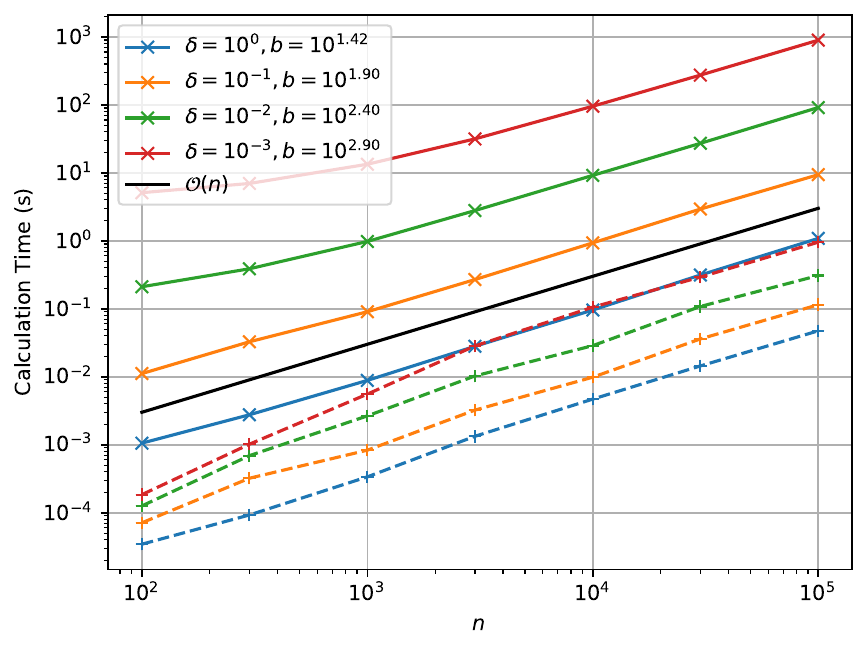}\\
\end{tabular}
\caption{Comparison of performance between the ${\cal O}(b^2n)$ algorithms of~\cite{Gutleb-Olver-Slevinsky-1-24} (solid lines) and the ${\cal O}(bn)$ algorithms in the present work (dashed lines). In both plots, the dashed lines correspond to the same $\delta$ value as the solid lines of the same colour. Left: the relative error measured in the Frobenius norm $\norm{P_nW_PP_n^\top - (P_nRP_n^\top)^\top P_nRP_n^\top}_F/\norm{P_nW_PP_n^\top}_F$. Right: calculation times of the construction of the Legendre--Gram matrix and its Cholesky factorization for four different values of $\delta$. The difference in the bandwidth scalings results in gaps between the solid lines being twice as wide as those between the dashed lines.}
\label{fig:legendregenerating}
\end{center}
\end{figure}

\subsection{Multiple algebraic factors on $(-1,1)$}\label{subsection:modifiedchebyshev}

In~\cite{Ellison-Julien-489-112268-23,Papadopoulos-et-al-46-A3448-24}, orthogonal polynomials are constructed with respect to algebraic weights with singularities at the endpoints and off the interval $(-1,1)$. While algebraic singularities off the interval can be well approximated by polynomials and rationals, when an integrable singularity is inside the interval, the methodology of~\cite{Gutleb-Olver-Slevinsky-1-24} breaks down. Here, we find orthogonal polynomials with respect to the weight:
\begin{equation}\label{eq:modifiedchebyshevweight}
w(x) = |x-\tfrac{1}{2}|^{-\frac{1}{2}}|x-\tfrac{1}{4}|^{-\frac{1}{4}}|x+\tfrac{1}{4}|^{\frac{1}{4}}|x+\tfrac{1}{2}|^{\frac{1}{2}}.
\end{equation}
Our algorithm is as follows:
\begin{enumerate}
\item We subdivide the unit interval as $[-1,-\tfrac{1}{2}]\cup[-\tfrac{1}{2},-\tfrac{1}{4}]\cup[-\tfrac{1}{4},\tfrac{1}{4}]\cup[\tfrac{1}{4},\tfrac{1}{2}]\cup[\tfrac{1}{2},1]$.
\item On each subinterval, we approximate the weight function with an algebraic endpoint-weighted mapped Chebyshev series. Each of these series converges geometrically by analyticity of $w(x)$ divided by the algebraic powers with singularities at the endpoints of each respective subinterval.
\item We compute the first four modified Chebyshev moments using this numerical approximation of $w(x)$.
\item We use Theorem~\ref{theorem:weightODE} to construct an at most nine-term recurrence relation for subsequent moments (that is degenerately $m$-term up to the $m^{\rm th}$ moment for $m<9$).
\item We fill the Chebyshev--Gram matrix by Theorem~\ref{theorem:momentstoX}, and Algorithm~\ref{algorithm:fastCholesky} computes the connection from Chebyshev polynomials to the desired ${\bf Q}(x)$.
\end{enumerate}

Figure~\ref{fig:modifiedchebyshev} illustrates the results of this algorithm by plotting the relative error in the computation of the first $10,000$ modified Chebyshev moments and by synthesizing the degree-$500$ polynomial.

\begin{figure}[htbp]
\begin{center}
\begin{tabular}{cc}
\includegraphics[width=0.48\textwidth]{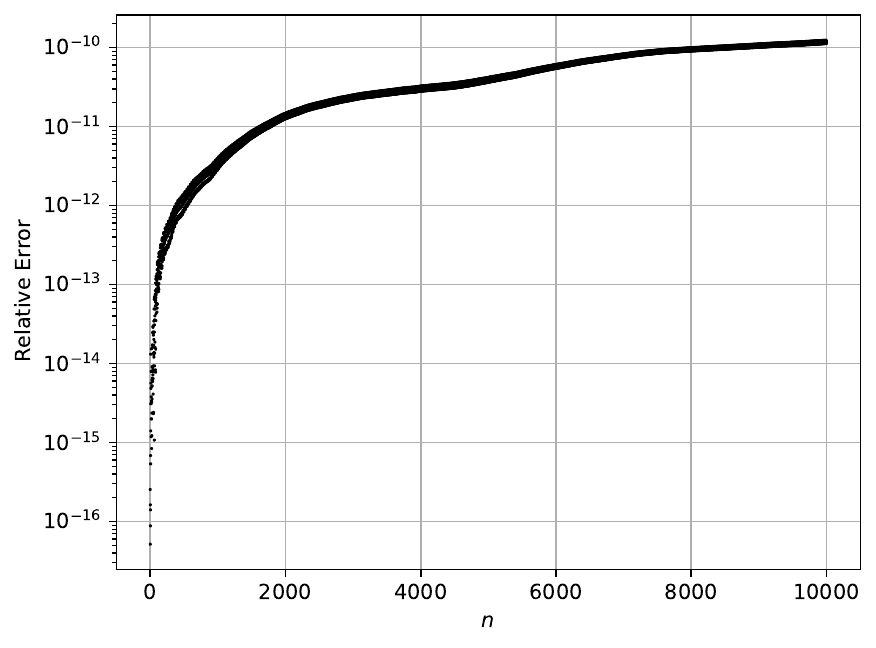}&
\includegraphics[width=0.48\textwidth]{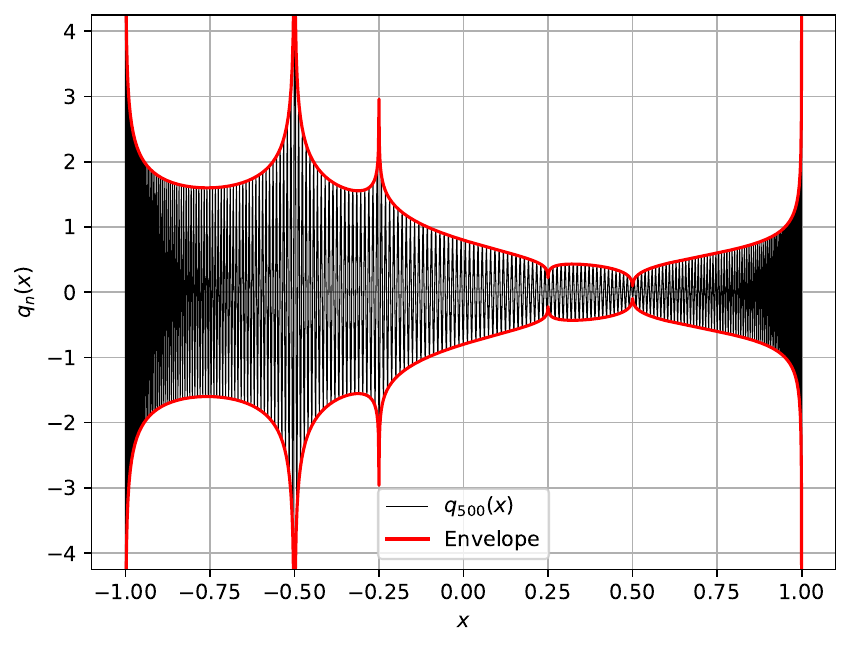}\\
\end{tabular}
\caption{A weight with multiple algebraic factors on $(-1,1)$. Left: the relative error in the modified Chebyshev moments. This is calculated by comparison with an extended precision calculation using MPFR~\cite{Fousse-et-al-33-13-1-07} following the same steps. Right: synthesis of $q_{500}(x)$ and the Szeg\H o evelope~\cite[Theorem 12.1.4]{Szego-75}.}
\label{fig:modifiedchebyshev}
\end{center}
\end{figure}

For this particular weight, the locations of the singularities are known in advance. More generally, the edge detection technique of~\cite{Pachon-Platte-Trefethen-30-898-10} can be used to determine the subdivision of the interval on-the-fly. In such a scenario, the differential structure may also be missing, but so long as there are a finite number of subintervals, the complexity of computing the first $2n$ moments would cost approximately ${\cal O}(n^2\log n)$ by the fast inverse discrete cosine transform (iDCT)~\cite{Frigo-Johnson-93-216-05}. Modified Chebyshev moments (and Jacobi more generally) can also be viewed within the framework of oscillatory integrals~\cite{Deano-Huybrechs-Iserles-17}. It is thus possible to develop rigorous asymptotic expansions that are more accurate for large degrees, facilitating an ${\cal O}(n)$ calculation the first $2n$ moments.

\subsection{A simple modified exponential on $\bbR^+$}

We close this section with the modified Laguerre weight:
\[
w(x) = \left[4,096\chi_{[0,4)}(x) + \chi_{[4,\infty)}(x)\right]e^{-x},
\]
and a word of caution. When working on unbounded domains, multiplication by $x$ is an unbounded matrix. This makes the calculation of the Gram matrix exceptionally difficult because both the Clenshaw algorithm~\cite{Clenshaw-9-118-55,Gutleb-Olver-Slevinsky-1-24} and the Gram matrix recurrence are unstable with unbounded multiplication matrices even when the Gram matrix is well-conditioned. Nevertheless, with an accurate Gram matrix in hand, its Cholesky factorization proceeds naturally in double precision. Here, we use Theorem~\ref{theorem:unboundedsimple} to compute the moments with respect to the standard Laguerre polynomials. In Table~\ref{table:degreetobits}, we report the number of bits required using MPFR's extended precision arithmetic~\cite{Fousse-et-al-33-13-1-07} to ensure that the $n\times n$ principal finite section of the Laguerre--Gram matrix agrees with the same matrix calculated using double the number of bits to $\epsilon_{\rm mach} \approx 2.22\times10^{-16}$ in the relative Frobenius norm. The number of bits are all multiples of $64$, consistent with the storage structure of the MPFR number type.

\begin{table}[htp]
\caption{Bits for an accurate Laguerre--Gram matrix $P_nW_LP_n^\top$.}
\begin{center}
\begin{tabular}{c|cccccc}
\hline
$n$ & $32$ & $64$ & $128$ & $256$ & $512$ & $1,024$\\
\hline
bits & $192$ & $256$ & $448$ & $896$ & $1,664$ & $3,328$\\
\hline
\end{tabular}
\end{center}
\label{table:degreetobits}
\end{table}

As can be seen in Table~\ref{table:degreetobits}, approximately three times as many bits are required as the truncation size. This is eventually problematic, and so we make another observation. With two accurate final rows and columns of $P_nW_LP_n^\top$, the recurrence relation in Eq.~\eqref{eq:MomentRecurrence} is stable in the downward direction. Suffice it to rapidly compute the entries in these last two columns! One might be tempted to linearize the product of Laguerre polynomials~\cite{Askey-131-70,Askey-Gasper-31-48-77}:
\[
L_m(x)L_n(x) = \sum_{k=|m-n|}^{m+n} \sum_{j\ge0}\frac{j!(-1)^{k+m+n}2^{k+m+n-2j}}{(j-k)!(j-m)!(j-n)!(k+m+n-2j)!}L_k(x).
\]
However, the factorial growth and alternation in sign in the coefficients combine to cause subtractive cancellation when used to express the final columns of the Laguerre--Gram matrix in terms of the modified Laguerre moments. It is possible that a fast Laguerre transform could be used to compute these coefficients in the right complexity; this is left for future exploration.

While this connection matrix is well-conditioned, with singular values contained in the interval $[1, 64]$, the eigenvalues suggest that the modified orthogonal polynomials multiplied by the square root of the weight will be localized to the interval $[0,4]$ at low degree or will treat the weight almost like the Laguerre weight at high degree, with a transition region in between. This phenomenon is plotted in Figure~\ref{fig:modifiedlaguerre}.

\begin{figure}[htbp]
\begin{center}
\begin{tabular}{cc}
\includegraphics[width=0.48\textwidth]{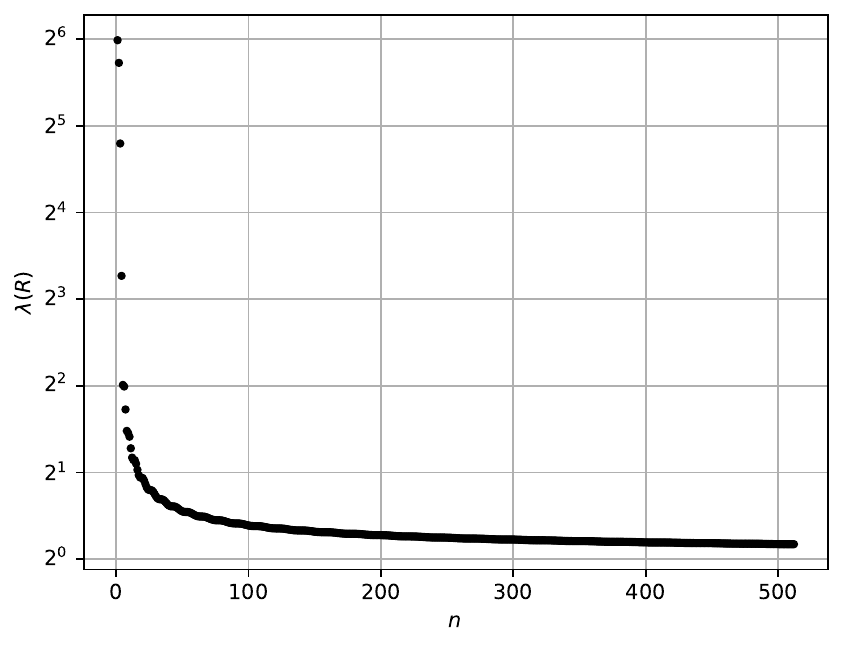}&
\includegraphics[width=0.48\textwidth]{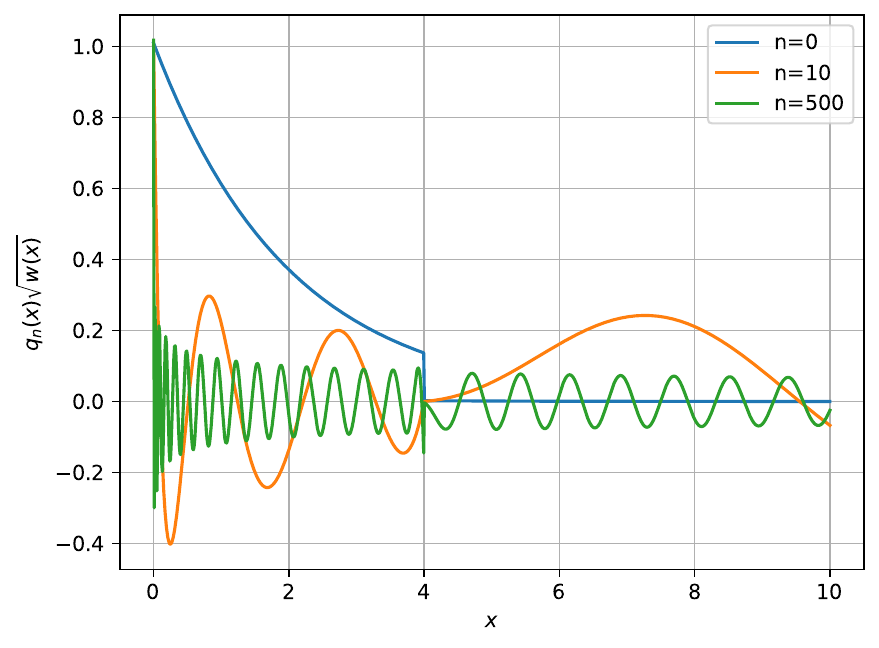}\\
\end{tabular}
\caption{Modified Laguerre polynomials. Left: the eigenvalues of the connection matrix $P_{512}RP_{512}^\top$. Right: three of the modified orthogonal polynomials multiplied by $\sqrt{w(x)}$. At low degree, this function is highly localized to the interval $[0,4]$, whereas at high degree, the function is not so different from $e^{-x/2}L_n(x)$, with a phase shift in the tail.}
\label{fig:modifiedlaguerre}
\end{center}
\end{figure}

\section{Hierarchical Cholesky factorization of the Gram matrix}\label{section:hierarchical}

We shall utilize the algorithm of Benner and Mach~\cite{Benner-Mach-439-1150-13} to provide accelerated matrix-vector products and solutions of linear systems with the connection coefficients under the assumption that the associated Gram matrix has a hierarchical off-diagonal low rank structure. If $W\in\bbR^{n\times n}$ is a symmetric positive definite matrix, it is an ${\cal H}_\ell$-matrix of rank $r$, denoted by ${\cal H}_\ell(r)$, if it can be expressed as:
\begin{equation}\label{eq:HODLRGram}
W = \begin{pmatrix} W_{11} & U\Sigma V^\top\\ V\Sigma^\top U^\top & W_{22}\end{pmatrix},
\end{equation}
where $\rank(U\Sigma V^\top) = r$ and $W_{11}$ and $W_{22}$ are ${\cal H}_{\ell-1}(r)$-matrices. In this context, $\ell$ refers to the number of levels in the hierarchical data structure that are used to store the matrix. With a conformably partitioned upper-triangular Cholesky factor, $R\in\bbR^{n\times n}$, the ${\cal H}_\ell(r)$-matrix structure is {\em not} preserved under Cholesky factorization:
\[
\begin{pmatrix} W_{11} & U\Sigma V^\top\\ V\Sigma U^\top & W_{22}\end{pmatrix} = \begin{pmatrix} R_{11}^\top\\ R_{12}^\top & R_{22}^\top\end{pmatrix}\begin{pmatrix} R_{11} & R_{12}\\ & R_{22}\end{pmatrix} = \begin{pmatrix} R_{11}^\top R_{11} & R_{11}^\top R_{12}\\ R_{12}^\top R_{11} & R_{12}^\top R_{12} + R_{22}^\top R_{22}\end{pmatrix}.
\]
While the Cholesky factorization of $W_{11}$ proceeds recursively, and it is plain to see that $\rank(R_{12}) \le r$, the Schur complement in the bottom right corner causes $R_{22}$ to be the Cholesky factorization of the rank-$r$ downdate of an ${\cal H}_{\ell-1}(r)$-matrix. Nevertheless, the Cholesky factor of an ${\cal H}_\ell(r)$-matrix is an ${\cal H}_\ell(r\ell)$-matrix~\cite[\S 3.4]{Benner-Mach-439-1150-13}, and compression of low-rank matrices in each level can keep the numerical ranks in the Cholesky factor as low as possible. By using the hierarchical matrix arithmetic~\cite{Borm-Grasedyck-Hackbusch-27-405-03}, it can be seen that the Cholesky factorization of an ${\cal H}_\ell(r)$-matrix costs ${\cal O}(r^2 n \log^2 n)$ flops and matrix-vector products with $R$ cost ${\cal O}(r n \log^2 n)$ flops. As will be seen, a mild $n$-dependent growth in the rank, such as $r = {\cal O}(\log n)$, can be included while preserving the quasi-linear asymptotic complexity of the Cholesky factorization. In this case, the complexity of the factorization is ${\cal O}(n \log^4n)$ and the matrix-vector products with $R$ cost ${\cal O}(n \log^3 n)$. We shall make one final observation. With compression of low-rank matrices at every step, it is often the case in our numerical experiments that the Cholesky factor is in fact an ${\cal H}_\ell(\hat{r})$-matrix, where $\hat{r}$ is more comparable to $r$ than to $r\ell$. If this assumption is made, then the complexity of the factorization more closely resembles ${\cal O}(n \log^3n)$ and the matrix-vector products then cost only ${\cal O}(n\log^2n)$ flops.

To use this hierarchical approach, we must approximate finite sections of the Gram matrix by Eq.~\eqref{eq:HODLRGram}. This question can be answered practically and theoretically. In the practical sense, our implementation uses randomized numerical linear algebra~\cite{Liberty-et-al-104-20167-07,Halko-Martinsson-Tropp-53-217-11,Martinsson-Tropp-29-403-20} to compute, with high probability, low-rank approximations to the off-diagonal blocks in Eq.~\eqref{eq:HODLRGram} in the form of a partial singular value decomposition (SVD). The following theorem is but one of a collection of results describing how these algorithms work~\cite[\S 10]{Halko-Martinsson-Tropp-53-217-11}.

\begin{theorem}[Theorem 10.7 in~\cite{Halko-Martinsson-Tropp-53-217-11}] Let $A\in\bbR^{m\times n}$. For a rank-$r$ approximation and an oversampling parameter $p\ge4$, such that $r+p\le\min\{m,n\}$, let $\Omega\in\bbR^{n\times(r+p)}$ be a standard Gaussian sampling matrix. From the factorization $A\Omega = QR$, for all $u,t\ge1$:
\[
\|A-QQ^\top A\|_F \le \left(1+t\sqrt{\frac{3r}{p+1}}\right)\sqrt{\sum_{j>r}\sigma_j(A)^2} + ut\frac{e\sqrt{r+p}}{p+1}\sigma_{r+1}(A),
\]
with failure probability at most $2t^{-p} + e^{-u^2/2}$.
\end{theorem}
A reduced $QR$ factorization can be converted into one of several other factorizations, including the partial SVD. A crucial observation of Xia and Gu~\cite[Proposition 2.1]{Xia-Gu-31-2899-10} states that when computing $R_{22}$ by $R_{22}^\top R_{22} = W_{11} - R_{12}^\top R_{12}$, the approximation $R_{22}$ is {\em guaranteed to exist} when approximating the off-diagonal blocks of $P_nW_PP_n^\top$ by a partial SVD. This is because the downdating step with a partial SVD is even more positive-definite than the full downdate with the dense off-diagonal block. In general, therefore, randomized algorithms will require ${\cal O}(n^2)$ flops to approximate $P_nW_PP_n^\top \approx W$ as an ${\cal H}_\ell(r)$-matrix. But there is at least one family of classical orthogonal polynomials for which this complexity can be reduced. By the Chebyshev linearization formula~\cite[\S 18.18.21]{NIST:DLMF}, $2T_m(x)T_n(x) = T_{m+n}(x) + T_{|m-n|}(x)$, entries of the Chebyshev--Gram matrix are the average of two modified Chebyshev moments:
\begin{equation}\label{eq:ChebyshevGram}
2(W_T)_{m,n} = \mu_{m+n}[w] + \mu_{|m-n|}[w].
\end{equation}
By the structure of the indexing of the modified moments, we would call $W_T$ a symmetric Toeplitz-plus-Hankel matrix, a property preserved under canonical orthogonal projection. By the fast Fourier transform (FFT)~\cite{Frigo-Johnson-93-216-05}, matrix-vector products with any $m\times n$ contiguous subblock of the Chebyshev--Gram matrix cost ${\cal O}(N\log N)$ flops, where $N = \max\{m,n\}$. This allows the randomized numerical linear algebra to produce low-rank approximations of contiguous subblocks in ${\cal O}(r n\log^2n)$ flops.

For a theoretical understanding, that is, a rigorous bound on the numerical rank of the off-diagonal subblocks, the symmetric Toeplitz-plus-Hankel structure of $W_T$ is also useful. Certain modified Chebyshev moments are known in closed form. By the popularity of the Clenshaw--Curtis quadrature rule, it is well known~\cite{Trefethen-50-67-08} that:
\begin{equation}\label{eq:ClenshawCurtismoments}
\mu_n[1] = \left\{\begin{array}{ccc} \frac{2}{1-n^2} & {\rm for} & n{\rm~even,}\\0 & {\rm for} & n{\rm~odd.}\end{array}\right.
\end{equation}
Analysis of these modified moments would reveal the rank structure behind the Chebyshev--Legendre transform, which was discovered by Alpert and Rokhlin~\cite{Alpert-Rokhlin-12-158-91} and refined by Keiner~\cite{Keiner-31-2151-09}. For a nonclassical transform, we will consider instead the log-Chebyshev weight~\cite[Eq.~(65)]{Slevinsky-Olver-332-290-17}:
\[
\mu_n\left[\log\left(\frac{2}{1-x}\right)\frac{1}{\sqrt{1-x^2}}\right] = \left\{\begin{array}{ccc} 2\pi\log2 & {\rm for} & n=0,\\ \frac{\pi}{n} & \multicolumn{2}{c}{\rm otherwise.}\end{array}\right.
\]
It can be shown by Theorem~\ref{theorem:algebraicmomentsequalslowrank} that the off-diagonal blocks have numerical rank ${\cal O}(\log n)$, confirmed in Figure~\ref{fig:hierarchicalcholesky}, where the numerical ranks in all subblocks of the Chebyshev--Gram matrix and its Cholesky factor are illustrated. Consequently, the error and calculation times of Figure~\ref{fig:hierarchicaltimeanderror} are to be expected.

\begin{figure}[htbp]
\begin{center}
\begin{tabular}{cc}
\includegraphics[width=0.475\textwidth]{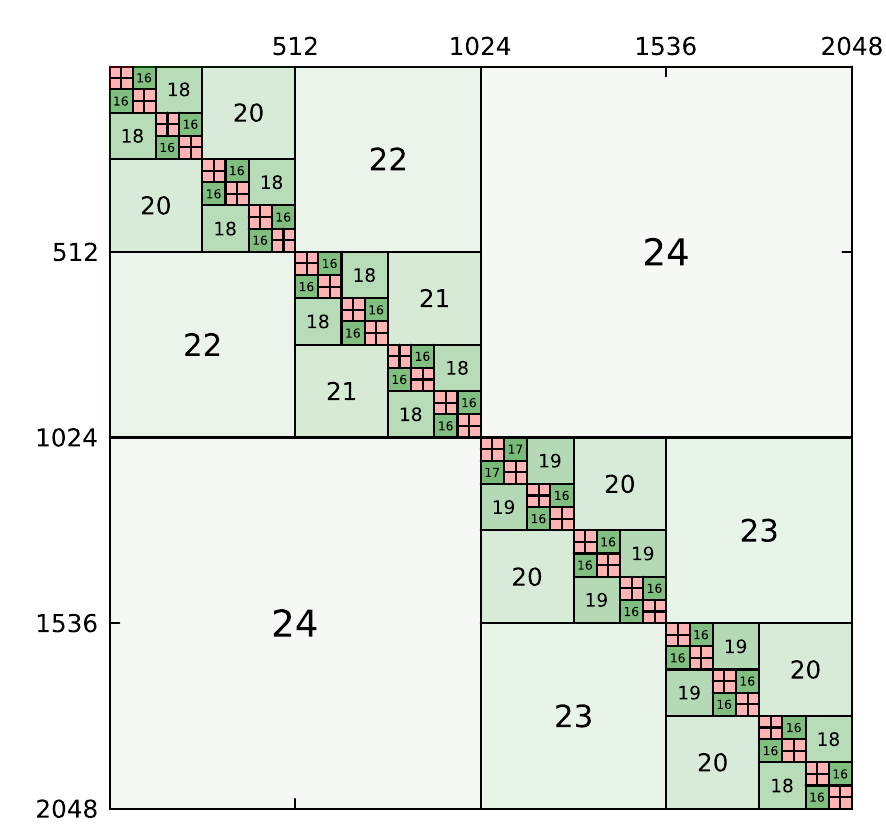}&
\includegraphics[width=0.475\textwidth]{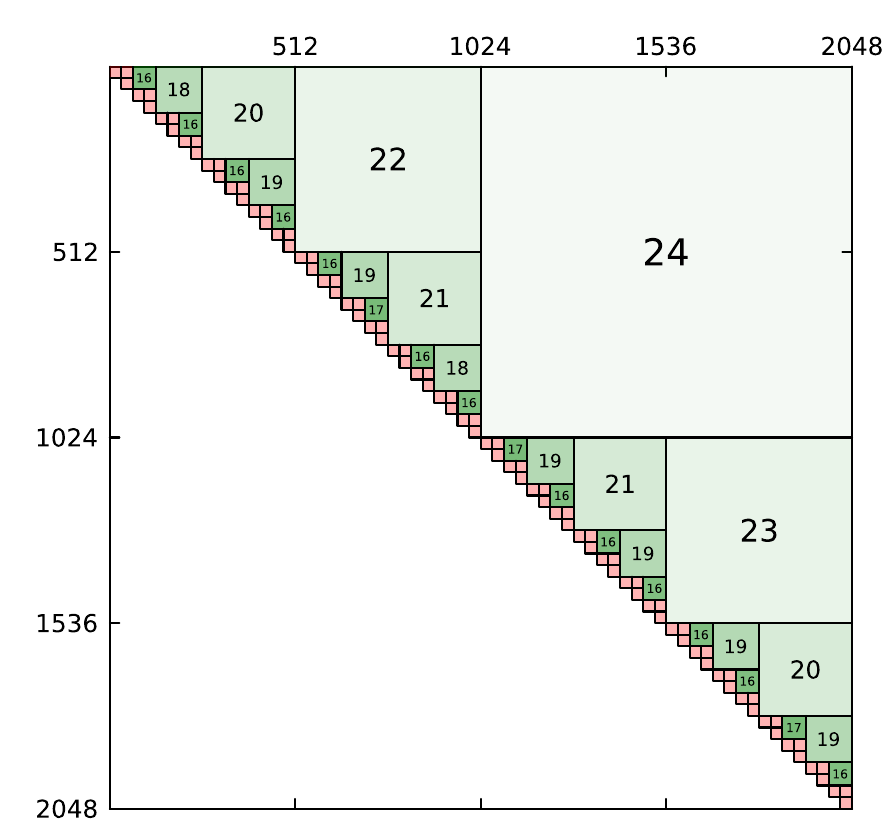}\\
\end{tabular}
\caption{Left: Numerical ranks of the hierarchical approximation of the principal finite section of the Chebyshev--Gram matrix for the log-Chebyshev weight $w(x) = \log\left(\dfrac{2}{1-x}\right)\dfrac{1}{\sqrt{1-x^2}}$. Right: numerical ranks of the corresponding Cholesky factor. In both panels, green blocks indicate low-rank approximations and red indicates that dense blocks are used. The opacity $\alpha\in(0,1)$ of a green block of size $m\times n$ and numerical rank $r$ is proportional to its data-sparsity compared to a dense filling: $\alpha mn = (m+n)r$.}
\label{fig:hierarchicalcholesky}
\end{center}
\end{figure}

\begin{figure}[htbp]
\begin{center}
\begin{tabular}{cc}
\includegraphics[width=0.48\textwidth]{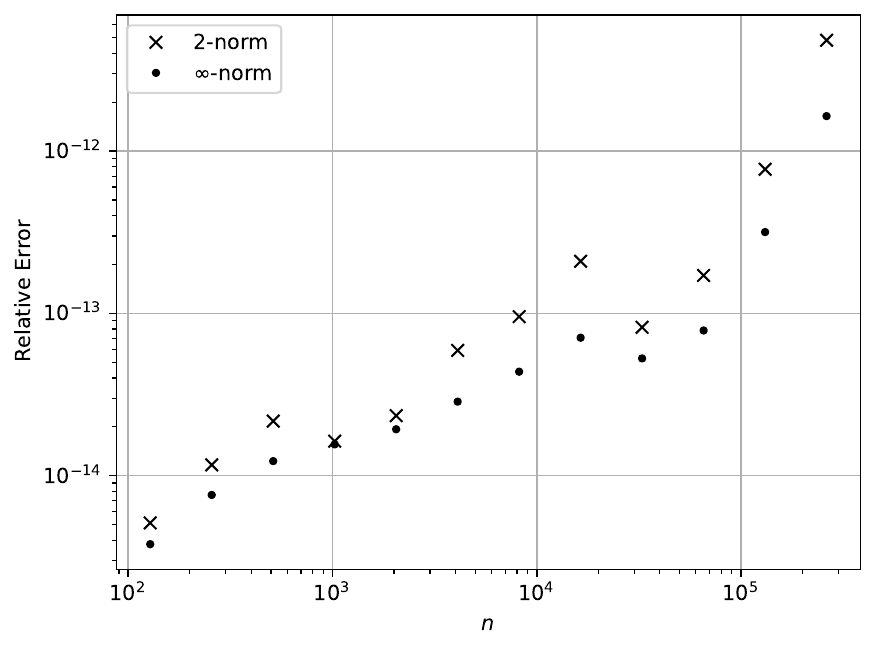}&
\includegraphics[width=0.48\textwidth]{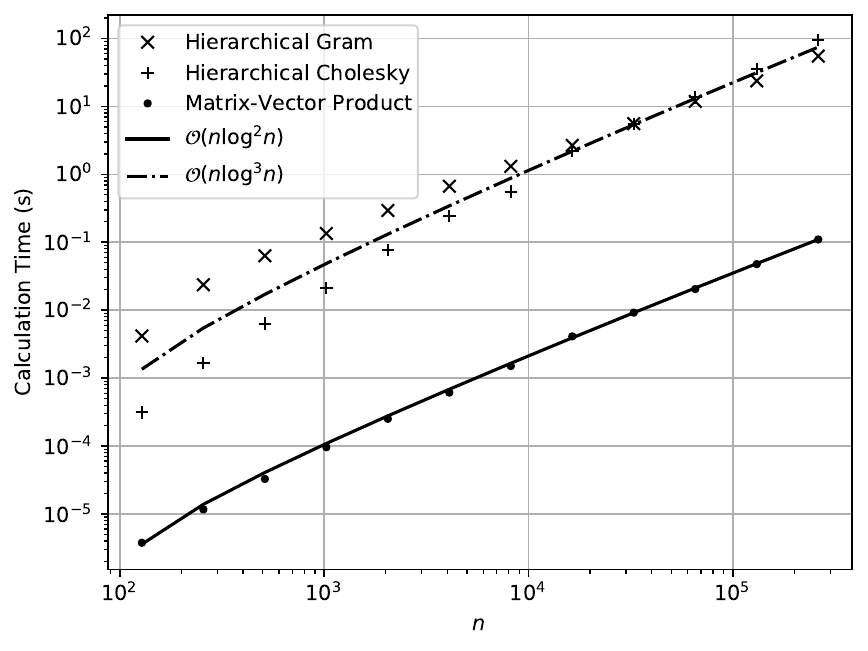}\\
\end{tabular}
\caption{Conversion of a degree-$n-1$ expansion in log-Chebyshev OPs with standard normally distributed pseudorandom coefficients to Chebyshev polynomials. Left: $2$-norm and $\infty$-norm relative error in the forward and backward transformation. Right: calculation times of the hierarchical construction of the Chebyshev--Gram matrix, its hierarchical Cholesky factorization, and the matrix-vector product with the hierarchical Cholesky factor.}
\label{fig:hierarchicaltimeanderror}
\end{center}
\end{figure}

While we are not aware of a universal proof that all modified Chebyshev polynomials can be well-approximated by an ${\cal H}_\ell$-matrix approach, we return to the example of \S~\ref{subsection:modifiedchebyshev} to illustrate through Figure~\ref{fig:hierarchicalcholeskybis} that this particular Chebyshev--Gram matrix also has a logarithmic growth in the ranks of the off-diagonal subblocks, a growth that is nearly preserved under Cholesky factorization.

\begin{figure}[htbp]
\begin{center}
\begin{tabular}{cc}
\includegraphics[width=0.475\textwidth]{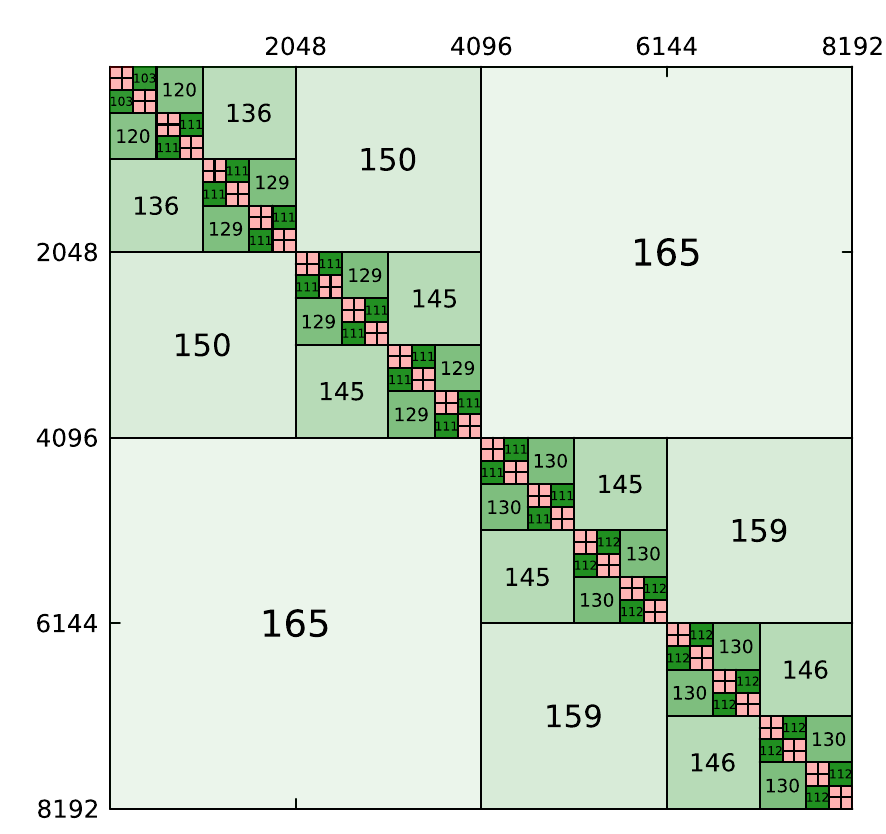}&
\includegraphics[width=0.475\textwidth]{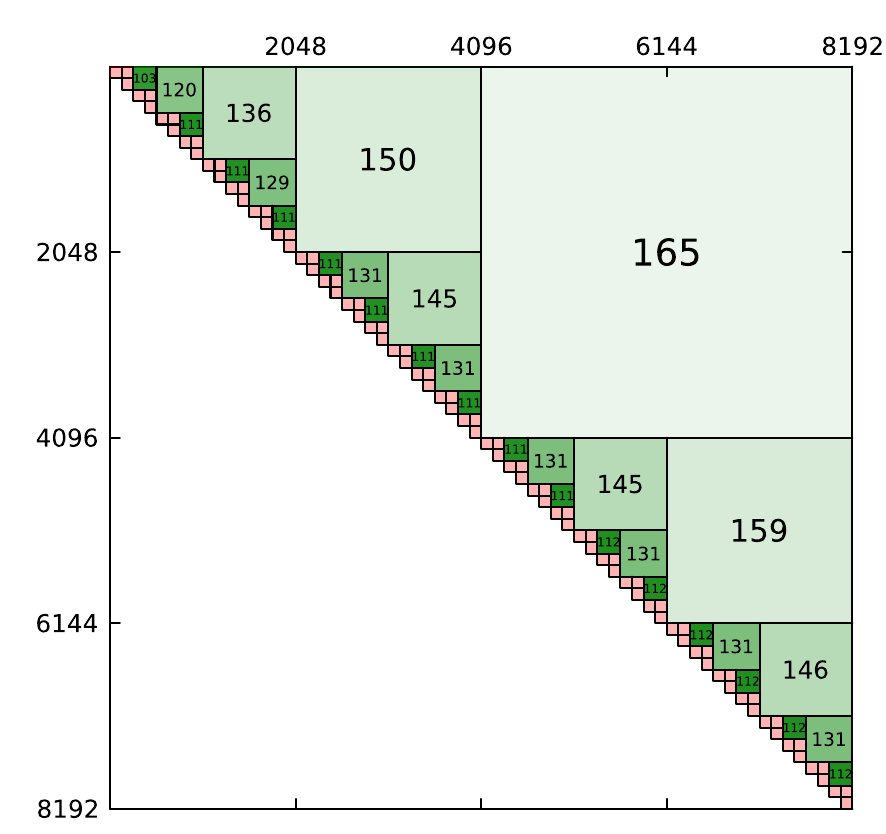}\\
\end{tabular}
\caption{Left: Numerical ranks of the hierarchical approximation of the principal finite section of the Chebyshev--Gram matrix for the weight in Eq.~\eqref{eq:modifiedchebyshevweight}. Right: numerical ranks of the corresponding Cholesky factor. The colours and opacities are produced in the same way as in Figure~\ref{fig:hierarchicalcholesky}.}
\label{fig:hierarchicalcholeskybis}
\end{center}
\end{figure}

\section{Conclusions \& future explorations}\label{section:conclusion}

We describe fast algorithms to produce and factor the Gram matrix that makes working with modified orthogonal polynomials numerically a breeze. There are a few promising extensions to consider. In the multivariate setting, OPs in $\bbR^d$ are characterized by a collection of $d$ multiplication matrices, such as multiplication by each of the Cartesian coordinates. Thus, the multivariate Gram matrices satisfy $d$ {\em simultaneous} matrix equations, and it remains to be seen how these can be combined to produce a fast Cholesky factorization. For a bivariate degree-$b$ polynomial measure modification of total degree-$n$, the direct banded Cholesky cost is ${\cal O}[(b n)^2n^2] = {\cal O}(b^2n^4)$, though the Cholesky factor contains only ${\cal O}(bn^3)$ nontrivial entries. Can displacement algorithms be used to produce the Cholesky factor in the optimal complexity? In Figure~\ref{fig:hierarchicaltimeanderror}, the execution time of the hierarchical Cholesky factor's matrix-vector product is quite competitive but the construction and factorization costs trail by about three orders of magnitude. Can the costs of these precomputations be reduced? With Sobolev orthogonal polynomials~\cite{Marcellan-Xu-33-308-15}, the inner products with derivatives no longer satisfy the simple matrix equation $X^\top W = W X$, but perhaps another matrix equation may be found to connect Sobolev polynomials to OPs with a standard inner product.

\bibliographystyle{siamplain}
\bibliography{/Users/Mikael/Bibliography/Mik}

\appendix

\section{Chebyshev polynomials}

The Chebyshev polynomials of the first kind, $T_n(x) = \cos[n\cos^{-1}(x)]$, are orthogonal polynomials on $(-1,1)$ with respect to the classical weight $w_c(x) = \frac{1}{\sqrt{1-x^2}}$. Next, we include a number of propositions that provide a self-contained description of all the formulas required to build the recurrence relations for the modified Chebyshev moments according to Theorem~\ref{theorem:weightODE}.

\begin{proposition}[\S 18.9.1 and 18.9.2 in~\cite{NIST:DLMF}]\label{prop:XT}
Multiplication by $x$ of first kind Chebyshev polynomials:
\begin{equation}
x{\bf T}(x) = {\bf T}(x)\underbrace{\begin{pmatrix} 0 & \tfrac{1}{2}\\ 1 & 0 & \tfrac{1}{2}\\ & \tfrac{1}{2} & 0 & \ddots\\ & & \ddots & \ddots\end{pmatrix}}_{X_T}.
\end{equation}
\end{proposition}

\begin{proposition}[Table 18.3.1 in~\cite{NIST:DLMF}]\label{prop:MT}
The first kind Chebyshev polynomial mass matrix:
\begin{equation}
\int_{-1}^1 {\bf T}^\top(x) {\bf T}(x)\frac{{\rm d}x}{\sqrt{1-x^2}} = \underbrace{\begin{pmatrix} \pi\\ & \frac{\pi}{2}\\ & & \tfrac{\pi}{2}\\ & & & \ddots\end{pmatrix}}_{M_T}.
\end{equation}
\end{proposition}

\begin{proposition}[\S 18.9.9 in~\cite{NIST:DLMF}]\label{prop:RTU}
Conversion of first kind Chebyshev polynomials:
\begin{equation}
{\bf T}(x) = {\bf U}(x)\underbrace{\begin{pmatrix} 1 & 0 & -\tfrac{1}{2}\\ & \tfrac{1}{2} & 0 & \ddots\\ & & \tfrac{1}{2} & \ddots\\ & & & \ddots\end{pmatrix}}_{R_T^U}.
\end{equation}
\end{proposition}

\begin{proposition}[\S 18.9.10 in~\cite{NIST:DLMF}]\label{prop:LUT}
Weighted conversion of second kind Chebyshev polynomials:
\begin{equation}
(1-x^2){\bf U}(x) = {\bf T}(x)\underbrace{\begin{pmatrix} \tfrac{1}{2}\\ 0 & \tfrac{1}{2}\\ -\tfrac{1}{2} & 0 & \tfrac{1}{2}\\ & \ddots & \ddots & \ddots\end{pmatrix}}_{L_U^T}.
\end{equation}
\end{proposition}

\begin{proposition}[\S 18.9.21 in~\cite{NIST:DLMF}]\label{prop:DTU}
Differentiation of first kind Chebyshev polynomials:
\begin{equation}
\DD {\bf T}(x) = {\bf U}(x)\underbrace{\begin{pmatrix} 0 & 1\\ & & 2\\ & & & 3\\ & & & & \ddots\end{pmatrix}}_{D_T^U}.
\end{equation}
\end{proposition}

\subsection{Modified Chebyshev moments}

Beyond the Clenshaw--Curtis and log-Chebyshev weights in \S~\ref{section:hierarchical}, some other modified Chebyshev moments are easy to describe, including the absolute value weight:
\[
\mu_n[|x|] = \left\{\begin{array}{ccc} \frac{1}{1-(n/2)^2} & {\rm for} & n {\rm~a~multiple~of~}4,\\ 0 & \multicolumn{2}{c}{\rm otherwise.}\end{array}\right.
\]
Next, we will use the pure log weight to illustrate how Theorem~\ref{theorem:weightODE} is useful in an analytical setting. Let $\psi(n)$ denote the digamma function, or the logarithmic derivative of the gamma function~\cite[\S 5.2.2]{NIST:DLMF} and let $\gamma = 0.577\ldots$ denote the Euler--Mascheroni constant~\cite[\S 5.2.3]{NIST:DLMF}.
\begin{lemma}
For $n>1$, the modified Chebyshev moments of the log weight are:
\[
(n^2-1)\mu_n\left[\log\left(\frac{2}{1-x}\right)\right] = \left\{\begin{array}{ccc} 2-2\gamma-4\ln2 - \frac{4n}{n^2-1}-2\psi(\frac{n-1}{2}) & {\rm for} & n{\rm~even,}\\ 2-2\gamma-4\ln2-\frac{2}{n}-2\psi(\frac{n}{2}) & {\rm for} & n{\rm~odd.}\end{array}\right.
\]
\end{lemma}
\begin{proof}
By Table~\ref{table:weightODEs}, the weight function $w(x) = \log\left(\dfrac{2}{1-x}\right)$ satisfies the inhomogeneous differential equation:
\[
(1-x^2)w' = [(1-x^2)w]'+2xw = 1+x.
\]
By Theorem~\ref{theorem:weightODE} and Propositions~\ref{prop:XT},~\ref{prop:MT},~\ref{prop:LUT}, and~\ref{prop:DTU}, this equation can be converted into the following three-term recurrence relation:
\begin{align*}
\mu_0[w] & = 2,\\
\mu_1[w] & = 1,\\
(n+2)\mu_{n+1}[w] - (n-2)\mu_{n-1}[w] & = 2\mu_n[1+x] = \left\{\begin{array}{ccc} \frac{4}{1-n^2} & {\rm for} & n{\rm~even},\\ \frac{4}{4-n^2} & {\rm for} & n{\rm~odd}.\end{array}\right.
\end{align*}
Separating even and odd modified moments, we find two-term recurrence relations for each of these. For even moments:
\[
(2n+3)\mu_{2n+2}[w] - (2n-1)\mu_{2n}[w] = \frac{4}{4-(2n+1)^2},
\]
and for odd moments:
\[
(2n+4)\mu_{2n+3}[w] - 2n\mu_{2n+1}[w] = \frac{1}{4-(2n+2)^2}.
\]
In the former:
\[
(1-4n^2)\mu_{2n}[w] = 2-\sum_{k=0}^{n-1}\frac{4(2k+1)}{4-(2k+1)^2},
\]
and in the latter:
\[
2n(2n+2)\mu_{2n+1}[w] = \sum_{k=0}^{n-1}\frac{4(2k+2)}{1-(2k+2)^2}.
\]
\end{proof}

As can be seen by the explicit examples, many modified Chebyshev moments have slow algebraic decay as a function of the degree of the Chebyshev polynomial. In this case, we develop of theory of low-rank approximation for such an {\em ansatz}. This result is similar to the constructions of~\cite{Alpert-Rokhlin-12-158-91,Grasedyck-13-01,Keiner-31-2151-09}.

\begin{theorem}\label{theorem:algebraicmomentsequalslowrank}
Suppose $\mu_n = \frac{2}{n^\alpha}$ for some $\alpha>0$ and for every $n>0$. Then, for $s =\lfloor\frac{n}{2}\rfloor$, there exists a rank-$r$ matrix $\tilde{W}_T$ such that:
\[
|(W_T)_{j,k} - (\tilde{W}_T)_{j,k}| \le \epsilon,\quad{\rm for}\quad 0\le j< s,\quad s\le k< n,
\]
where:
\[
r = r(\alpha, \epsilon, n) = \log\left[\frac{\epsilon}{2}\left(\frac{(s-1)(1-z(s))^2}{4z(s)}\right)^\alpha\right]\Bigg/\log\left(\frac{2\alpha e z(s)}{1-z(s)}\right),
\]
where $z(s) = \frac{s-1}{s+1+2\sqrt{s}}$.
\end{theorem}
\begin{proof}
By Eq.~\eqref{eq:ChebyshevGram}:
\[
(W_T)_{j,k} = \frac{1}{(k-j)^\alpha} + \frac{1}{(k+j)^\alpha},\quad{\rm for}\quad k>j.
\]
We shall create a polynomial approximation to a fixed column $s\le k<n$. Consider the map $m : x \mapsto \frac{(s-1)(1+x)}{2}$ from the unit interval $[-1,1]$ to $[0,s-1]$. We will consider the function:
\[
W_T(x;k) = \frac{1}{\left(k-\frac{(s-1)(1+x)}{2}\right)^\alpha} + \frac{1}{\left(k+\frac{(s-1)(1+x)}{2}\right)^\alpha}.
\]
By the ultraspherical polynomial generating function~\cite[\S 18.12.4]{NIST:DLMF}:
\[
\frac{1}{\left(\dfrac{z+z^{-1}}{2}-x\right)^\lambda} = |2z|^\lambda\sum_{n=0}^\infty C_n^{(\lambda)}(x) z^n,\quad{\rm for}\quad \lambda>0,~|z|<1,~|x|\le 1,
\]
we can find a series representation for $W_T(x;k)$:
\begin{align*}
& \left(\frac{s-1}{2}\right)^\alpha W_T(x;k) = \frac{1}{\left(\frac{2k}{s-1}-1-x\right)^\alpha} + \frac{1}{\left(\frac{2k}{s-1}+1+x\right)^\alpha},\\
& = [2z_1(k,s)]^\alpha\sum_{n=0}^\infty C_n^{(\alpha)}(x) z_1(k,s)^n + [2z_2(k,s)]^\alpha\sum_{n=0}^\infty C_n^{(\alpha)}(-x) z_2(k,s)^n,
\end{align*}
where $z_1(k,s)$ and $z_2(k,s)$ are the smallest (positive) respective roots of:
\begin{align*}
z_1^2 - \left(\frac{4k}{s-1}-2\right)z_1 + 1 & = 0,\quad{\rm and}\\
z_2^2 -\left(\frac{4k}{s-1}+2\right)z_2 + 1 & = 0.
\end{align*}
Since the slopes of the linear terms dictate that $z_2(k,s) < z_1(k,s)<1$, the error incurred in truncating the series after $N$ terms is:
\begin{align*}
& \left|W_T(x;k) - \left(\frac{4z_1(k,s)}{s-1}\right)^\alpha \sum_{n=0}^{N-1} C_n^{(\alpha)}(x) z_1(k,s)^n - \left(\frac{4z_2(k,s)}{s-1}\right)^\alpha \sum_{n=0}^{N-1} C_n^{(\alpha)}(-x) z_2(k,s)^n \right|,\\
& \le \left(\frac{4z_1(k,s)}{s-1}\right)^\alpha \sum_{n=N}^\infty \left|C_n^{(\alpha)}(x)\right| z_1(k,s)^n + \left(\frac{4z_2(k,s)}{s-1}\right)^\alpha \sum_{n=N}^\infty \left|C_n^{(\alpha)}(-x)\right| z_2(k,s)^n,\\
& \le \left(\frac{4z_1(k,s)}{s-1}\right)^\alpha \sum_{n=N}^\infty \frac{(2\alpha)_n}{n!} z_1(k,s)^n + \left(\frac{4z_2(k,s)}{s-1}\right)^\alpha \sum_{n=N}^\infty \frac{(2\alpha)_n}{n!} z_2(k,s)^n,\\
& \le 2\left(\frac{4z_1(k,s)}{s-1}\right)^\alpha \sum_{n=N}^\infty \frac{(2\alpha)_n}{n!} z_1(k,s)^n,\\
& = 2\left(\frac{4z_1(k,s)}{s-1}\right)^\alpha \frac{(2\alpha)_N}{N!}z_1(k,s)^N \sum_{n=0}^\infty \frac{(2\alpha+N)_n}{(1+N)_n} z_1(k,s)^n,\\
& \le 2\left(\frac{4z_1(k,s)}{s-1}\right)^\alpha \frac{(2\alpha)_N}{N!}z_1(k,s)^N \sum_{n=0}^\infty \frac{(2\alpha+N)_n}{n!} z_1(k,s)^n,\\
& = 2\left(\frac{4z_1(k,s)}{s-1}\right)^\alpha \frac{(2\alpha)_N}{N!}\frac{z_1(k,s)^N}{(1-z_1(k,s))^{2\alpha+N}},
\end{align*}
where we have used the uniform upper bound on ultraspherical polynomials~\cite[18.14.4]{NIST:DLMF}, $|C_n^{(\lambda)}(x)|\le \frac{(2\lambda)_n}{n!}$ for $|x|\le1$. By the binomial inequality:
\[
\frac{(2\alpha)_N}{N!} = \binom{N+2\alpha-1}{N} \le (2\alpha e)^N,
\]
the absolute error is less than or equal to $\epsilon$ when:
\[
N = \log\left[\frac{\epsilon}{2}\left(\frac{(s-1)(1-z_1(k,s))^2}{4z_1(k,s)}\right)^\alpha\right]\Bigg/\log\left(\frac{2\alpha e z_1(k,s)}{1-z_1(k,s)}\right).
\]
At worst:
\[
z_1(k,s) \le z_1(s,s) = z(s) = \frac{s-1}{s+1+2\sqrt{s}}.
\]
Thus, if we take $N = r(\alpha, \epsilon, n)$, we have secured a degree-$(r-1)$ polynomial that absolutely approximates every entry in any fixed column of the subblock. The matrix $\tilde{W}_T$ is formed by the outer product of two matrices: the first contains evaluations of the first $r$ ultraspherical polynomials at the points $-1+2j/(s-1)$ for $0\le j<s$. The second contains the $r$ coefficients from the generating functions for each column $s\le k< n$. By construction, $\rank(\tilde{W}_T) = r$.
\end{proof}

Notice that since $1-z(\lfloor\frac{n}{2}\rfloor) = {\cal O}(1/\sqrt{n})$, it follows that $r = {\cal O}(\log n)$ as $n\to\infty$. While this bound on the rank is by no means sharp, the logarithmic growth with respect to $n$ is rigorously achieved, and we leave the computation of the best numerical approximation to the algorithms from randomized numerical linear algebra. It is also trivial to convert this into a relative bound on the whole subblock by depressing $\epsilon$. If the moments require an even-odd piecewise definition, such as in Eq.~\eqref{eq:ClenshawCurtismoments}, the numerical rank is at most the sum of the numerical ranks of both pieces.

Finally, we conclude with bounds on the modified Chebyshev moments based on the smoothness of the weight, similar in spirit to Trefethen's~\cite[Theorem 7.1]{Trefethen-12}. These bounds enrich the picture on the near-universal decay rates of modified Chebyshev moments. Decay rates alone, however, do not imply low-rank approximation as can be seen by taking pseudorandom coefficients scaled by a prescribed decay. As in Theorem~\ref{theorem:algebraicmomentsequalslowrank}, usually stronger assumptions on the modified moments are necessary, such as analyticity of an underlying function that is sampled discretely to produce the moment vector.

\begin{definition}
The total variation of a function $f:[a,b]\to\bbR$ is:
\[
V_a^b(f) = \sup_{P\in{\cal P}}\sum_{i=0}^{n_P-1} | f(x_{i+1}) - f(x_i)|,
\]
where ${\cal P} = \{P=(x_0,\ldots,x_{n_P}) : a=x_0<\cdots<x_{n_P}=b\}$.
\end{definition}

\begin{definition}
A function $f:[a,b]\to\bbR$ is said to be of {\em bounded variation} if its total variation is finite.
\end{definition}

\begin{theorem}
Let $w:[-1,1]\to\bbR$ be of bounded variation. Then for $n\ge2$, the modified Chebyshev moments of $w$ satisfy:
\begin{equation}\label{eq:mufirstbound}
\mu_n[w] \le \frac{2\norm{w}_\infty}{(n-1)^2} + \frac{V_{-1}^1(w)}{n-1}.
\end{equation}
Moreover, if $w$ is absolutely continuous on $[-1,1]$ and $w'$ is of bounded variation, then for $n\ge3$, the modified Chebyshev moments of $w$ satisfy:
\begin{equation}\label{eq:musecondbound}
\mu_n[w] \le \frac{2\norm{w}_\infty}{(n-1)^2} + \frac{2\norm{w'}_\infty}{(n-2)^3} + \frac{V_{-1}^1(w')}{(n-2)^2}.
\end{equation}
\end{theorem}
\begin{proof}
By definition:
\begin{align*}
\mu_n[w] & = \int_{-1}^1 T_n(x) w(x){\rm\,d}x,\\
& = \left.\frac{w(x)}{2}\left(\frac{T_{n+1}(x)}{n+1}-\frac{T_{n-1}(x)}{n-1}\right)\right|_{-1}^1 - \frac{1}{2}\int_{-1}^1\left(\frac{T_{n+1}(x)}{n+1}-\frac{T_{n-1}(x)}{n-1}\right)w'(x){\rm\,d}x,\\
& = \frac{w(1)+(-1)^nw(-1)}{2}\frac{2}{1-n^2} - \frac{1}{2}\int_{-1}^1\left(\frac{T_{n+1}(x)}{n+1}-\frac{T_{n-1}(x)}{n-1}\right)w'(x){\rm\,d}x.
\end{align*}
Then~\eqref{eq:mufirstbound} follows from:
\[
\mu_n[w] \le \frac{2\norm{w}_\infty}{n^2-1} + \frac{n\norm{w'}_1}{n^2-1},
\]
where the derivative of $w$ may be interpreted in a distributional sense if necessary~\cite[Chapter 5]{Ziemer-89}.
As for~\eqref{eq:musecondbound}, since:
\[
\mu_n[w'] \le \frac{2\norm{w'}_\infty}{n^2-1} + \frac{n\norm{w''}_1}{n^2-1},
\]
and:
\[
\mu_n[w] = \frac{w(1)+(-1)^nw(-1)}{2}\frac{2}{1-n^2} - \frac{1}{2}\left(\frac{\mu_{n+1}[w']}{n+1} - \frac{\mu_{n-1}[w']}{n-1}\right),
\]
it follows that:
\begin{align*}
\mu_n[w] & \le \frac{2\norm{w}_\infty}{n^2-1} + \frac{\norm{w'}_\infty}{(n+1)[(n+1)^2-1]} + \frac{\norm{w'}_\infty}{(n-1)[(n-1)^2-1]},\\
& \quad + \frac{n+1}{(n+1)^2-1}\norm{w''}_1\frac{1}{2(n+1)} + \frac{n-1}{(n-1)^2-1}\norm{w''}_1\frac{1}{2(n-1)},\\
& \le \frac{2\norm{w}_\infty}{n^2-1} + \frac{2\norm{w'}_\infty}{(n-1)[(n-1)^2-1]} + \frac{\norm{w''}_1}{(n-1)^2-1},\\
& \le \frac{2\norm{w}_\infty}{(n-1)^2} + \frac{2\norm{w'}_\infty}{(n-2)^3} + \frac{\norm{w''}_1}{(n-2)^2}.
\end{align*}

\end{proof}

\end{document}